\documentclass[preprint,3p,sort]{elsarticle}

\usepackage{amsfonts}
\usepackage{amsthm}
\usepackage{graphicx}
\usepackage{epstopdf}
\usepackage{algorithmic}
\usepackage{amsmath}
\usepackage{amssymb}
\usepackage{hyperref}
\usepackage[T1]{fontenc}
\usepackage{soul,color}
\newtheorem{theorem}{Theorem}[section]
\newtheorem{proposition}[theorem]{Proposition}

\newtheorem{definition}{Definition}[section]
\theoremstyle{remark}
\newtheorem{remark}[theorem]{Remark}
\theoremstyle{example}

\numberwithin{equation}{section}
\begin{document}


\begin{frontmatter}

%

\title{\textcolor{green}{A skew-symmetric energy and entropy} stable formulation of the compressible Euler equations}

\author[sweden,southafrica]{Jan Nordstr\"{o}m}
\cortext[secondcorrespondingauthor]{Corresponding author}
\ead{jan.nordstrom@liu.se}
\address[sweden]{Department of Mathematics, Applied Mathematics, Link\"{o}ping University, SE-581 83 Link\"{o}ping, Sweden}
\address[southafrica]{Department of Mathematics and Applied Mathematics, University of Johannesburg, P.O. Box 524, Auckland Park 2006, Johannesburg, South Africa}

\begin{abstract}
We show that a specific skew-symmetric form of nonlinear hyperbolic problems leads to energy and \textcolor{green}{entropy bounds. Next, we exemplify by considering the compressible Euler equations in primitive variables, transform them to skew-symmetric form and show how to obtain energy and entropy estimates. Finally we show that the skew-symmetric formulation lead to energy and entropy} stable discrete approximations if the scheme is formulated on summation-by-parts form. 
\end{abstract}

\begin{keyword}
Nonlinear hyperbolic problems \sep skew-symmetric form \sep compressible Euler equations \sep  energy stability \sep entropy stability  \sep summation-by-parts
\end{keyword}


\end{frontmatter}


\section{Introduction}

The energy method applied to linear  \textcolor{blue}{initial boundary value problems (IBVPs)} lead to well posed boundary conditions and energy estimates \cite{kreiss1970,kreiss1989initial,Gustafsson1978,gustafsson1995time,oliger1978,nordstrom2020,nordstrom_roadmap,nordstrom2005}. 
 \textcolor{blue}{The energy analysis uses integration-by-parts (IBP) as the main tool and leads to an estimate in an $L_2 $ equivalent norm. Symmetric matrices and a symmetrising matrix are normally required which can be hard to find in the nonlinear case (although exeptions exist \cite{nordstrom2021linear}).
The most common procedure to obtain estimates for nonlinear problems is to use the entropy stability theory \cite{godunov1961interesting,volpert1967,kruzkov1970,dafermos1973entropy,lax1973,harten1983,tadmor1984,Tadmor1987,Tadmor2003}.  
The entropy analysis needs both IBP and the chain rule and aims for conservation of mathematical entropy (a convex function more or less related to the physical entropy). In combination with certain sign requirements (for compressible flow the density and temperature must be positive) it leads to $L_2 $ estimates, otherwise not. It can with relative ease be applied to many nonlinear equations without specific symmetry requirements on the involved matrices.}

In this paper we will apply  \textcolor{blue}{and extend} the general stability theory developed in \cite{nordstrom2022linear-nonlinear}. This theory is valid for both linear and nonlinear problems and extend the use of the energy method to nonlinear problems. It is very direct, easy to understand and leads to $L_2 $ estimates. The only requirement for the energy bound is that a certain skew-symmetric form of the governing equations exist. It was shown in \cite{nordstrom2022linear-nonlinear} that this form exists for the velocity-divergence form of the incompressible Euler equations and could be derived for the shallow water equations  \textcolor{green}{(SWEs). One} drawback with this procedure is that the required skew-symmetric form can be quite complicated to derive. In this paper we will explain this procedure in detail \textcolor{green}{and use} the compressible Euler equations \textcolor{green}{as an example}. Once the skew-symmetric formulation is obtained, an energy bound follows by applying IBP, without using the chain rule and without sign requirements.  \textcolor{blue}{Although we focus on energy estimates, we show that the new formulation also allows for a mathematical (or generalised) entropy conservation and bound}. By discretising the equations \textcolor{green}{in space} using summation-by-parts (SBP) operators \cite{svard2014review,fernandez2014review}, nonlinear stability follows by discretely mimicking the IBP procedure. 

\textcolor{red}{Skew-symmetric formulations for parts or the whole set of governing flow equations have drawn interest previously. In \cite{Petr2007}, a similar set of variables, were considered. $L_2 $ continuous estimates were obtained for a subset of the variables. The discrete aspect was discussed but not analysed.  Aiming for simulation of turbulence, also \cite{Rozema2014} considered similar variables. Focus was on various conservation properties which were deduced from an a' priori assumption that an energy bound existed. The discrete conservation and stability aspects were discussed, and supporting calculations were provided, but no proofs were given. A related ambition using conventional primitive variables were provided in \cite{Sesterhenn2014}. Focus was again on conservation properties, in particular on preservation of kinetic energy in aeroacoustic calculations and supporting calculations were provided. Also in \cite{Halpern2018} primitive variables were used and skew-symmetry was targeted in order to preserved moments in plasma physics calculations. Neither in \cite{Sesterhenn2014} nor  \cite{Halpern2018}  were proofs provided.}  

\textcolor{green}{The papers \cite{Petr2007,Rozema2014,Sesterhenn2014,Halpern2018} (and references therein) contain fragments of the general theory for nonlinear hyperbolic problems presented in \cite{nordstrom2022linear-nonlinear} and in this paper. This theory include the compressible Euler equations which we use in this paper to exemplify the whole chain of actions leading to nonlinear stability.  As stated above, The only requirement for the continuous and discrete bounds is that a certain skew-symmetric form of the governing equations exist. We show in detail how to arrive at this formulation}. The remaining part of paper is organised as follows: In Section~\ref{sec:theory} we shortly reiterate the main theoretical findings in \cite{nordstrom2022linear-nonlinear} and outline the general procedure \textcolor{green}{for obtaining energy and entropy bounds. After that, we leave the general theory and we move to the compressible Euler equations which we use as the prime example}. We proceed in Section~\ref{sec:skew-eq} to choose an appropriate solution norm  on which we base our choice of new dependent variables. With a suitable form of the norm, we derive new governing equations in the new variables. Next we rewrite the governing equations in skew-symmetric form and show how to get an energy  \textcolor{blue}{and entropy} bound.  Section~\ref{numerics} illustrate the relation between the new continuous skew-symmetric formulation and stability of the numerical \textcolor{green}{SBP based} semi-discrete scheme. A summary and conclusions are provided in Section~\ref{sec:conclusion}.



\section{The main theory}\label{sec:theory}

The general theory extending the energy method to the nonlinear case in  \cite{nordstrom2022linear-nonlinear} is shortly summarised here. Consider the following general hyperbolic IBVP
\begin{equation}\label{eq:nonlin}
P U_t + (A_i(V) U)_{x_i}+B_i(V)U_{x_i}+C(V)U=0,  \quad t \geq 0,  \quad  \vec x=(x_1,x_2,..,x_k) \in \Omega
\end{equation}
augmented with homogeneous boundary conditions $L_p U=0$ at the boundary $\partial \Omega$. In (\ref{eq:nonlin}),  the Einstein summation convention is used and $P$ is a symmetric positive definite (or semi-definite) time-independent matrix that defines an energy norm (or semi-norm) $\|U\|^2_P= \int_{\Omega} U^T P U d\Omega$. We assume that $U$ and $V$ are smooth. The $n \times n$ matrices $A_i,B_i,C$ are smooth functions (each matrix element is smooth) of the $n$ component vector $V$, but otherwise arbitrary. Note that (\ref{eq:nonlin}) encapsulates both linear ($V \neq U$) and nonlinear  ($V=U$) problems. 

\textcolor{green}{\subsection{Energy analysis}\label{sec:energy}}
The following two concepts are essential for a proper treatment of (\ref{eq:nonlin}).
\begin{definition}
The problem (\ref{eq:nonlin}) is energy conserving if $\|U\|^2_P= \int_{\Omega} U^T P U d\Omega$ only changes due to boundary effects. It is energy bounded if $\|U\|^2_P < \infty$ as $t  \rightarrow  \infty$.
\end{definition}
\begin{proposition}\label{lemma:Matrixrelation}
The IBVP  (\ref{eq:nonlin}) for linear ($V \neq U$) and nonlinear  ($V=U$)  is energy conserving if
\begin{equation}\label{eq:boundcond}
B_i= A_i^T, \quad i=1,2,..,k \quad \text{and } \quad C+C^T = 0
\end{equation}
holds. It is energy bounded if it is energy conserving and the boundary conditions $L_p U=0$ are such that 
\begin{equation}\label{1Dprimalstab}
\oint\limits_{\partial\Omega}U^T  (n_i A_i)   \\\ U \\\ ds = \oint\limits_{\partial\Omega} \frac{1}{2} U^T ((n_i A_i)  +(n_i A_i )^T) U \\\ ds \geq 0,
\end{equation}
\end{proposition}
\begin{proof}
The energy method applied to (\ref{eq:nonlin}) yields
\begin{equation}\label{eq:boundaryPart1}
\frac{1}{2} \frac{d}{dt}\|U\|^2_P + \oint\limits_{\partial\Omega}U^T  (n_i A_i)  \\\ U \\\ ds= \int\limits_{\Omega}(U_{x_i}^T  A_i U - U^T B_i U_{x_i}) \\\ d \Omega -\int\limits_{\Omega} U^T  C U \\\ d \Omega,
\end{equation}
where $(n_1,..,n_k)^T$ is the outward pointing unit normal. The terms on the right-hand side of (\ref{eq:boundaryPart1}) are cancelled by (\ref{eq:boundcond}) leading to energy conservation. If in addition (\ref{1Dprimalstab}) holds, an energy bound is obtained. 
\end{proof}
\begin{remark}\label{list}
Proposition \ref{lemma:Matrixrelation}  shows that whatever form the original IBVP has, energy boundedness and energy conservation can be proved if it can be rewritten in the form given by (\ref{eq:nonlin})-(\ref{eq:boundcond}).
The procedure to arrive at an energy estimate involve the following steps.
\begin{enumerate}
\item Find an appropriate energy norm from which one can choose new dependent variables.
\item Derive a new set of governing equations in the new variables from the standard governing equations.
\item Transform the new set of equations into a skew-symmetric formulation as in (\ref{eq:nonlin}).
\item Apply the energy method such that only boundary terms remain as in (\ref{eq:boundaryPart1}).
\item Find a minimal number of boundary conditions that limits the resulting boundary terms as in (\ref{1Dprimalstab}).
\end{enumerate}
\end{remark}
\begin{remark}\label{necessary}
\textcolor{blue}{For linear problems, a minimal number of boundary conditions that lead to a bound is a necessary and sufficient condition for well-posedness. For nonlinear problems this is not the case. A minimal number of boundary conditions that lead to a bound is a necessary but not a sufficient condition \cite{nordstrom2005,nordstrom2020,nordstrom2021linear}.}
\end{remark}
\textcolor{green}{We will in Section \ref{sec:skew-eq} show a detailed derivation of steps 1-3 for the compressible Euler equations. Steps 4-5 will be shortly reviewed. Before that, we will show that also a specific mathematical entropy is conserved.}

\textcolor{blue}{\subsection{Entropy analysis}}\label{sec:entropy}
\textcolor{blue}{For non-smooth nonlinear solutions $U$, the formulation (\ref{eq:nonlin}) interpreted in a weak sense also allows for an entropy conservation law.
\begin{proposition}\label{lemma:Entropy-relation}
The IBVP  (\ref{eq:nonlin}) together with the conditions (\ref{eq:boundcond}) leads to the entropy conservation law
\begin{equation}\label{eq:entropy-equation}
S_t + (F_i)_{x_i} = 0,
\end{equation}
where $S=U^T P U/2$ is the mathematical (or generalised) entropy  and  $F_i=U^T A_i U$ are the entropy fluxes.
\end{proposition}
\begin{proof}
Multiplication of  (\ref{eq:nonlin}) from the left with $U^T$ yields
\begin{equation}\label{eq:entropy-der}
 (U^T P U/2)_t  +  (U^T A_i U)_{x_i}  = (U_{x_i}^T  A_i U - U^T B_i U_{x_i}) - U^T  C U.
\end{equation}
The right-hand side of (\ref{eq:entropy-der}) is cancelled by (\ref{eq:boundcond}) leading to the entropy conservation relation (\ref{eq:entropy-equation}).
\end{proof}
\begin{remark}\label{entropy-comments}
The entropy conservation law (\ref{eq:entropy-equation}) holds for smooth solutions. For discontinuous solutions it holds in a distributional sense. The non-standard compatibility conditions in this case reads
\begin{equation}\label{eq:entropy-der}
\partial S/\partial U=S_{U}=U^T P,  \quad S_{U} P^{-1}( (A_i(V) U)_{x_i}+B_i(V)U_{x_i}+C(V)U)=(U^T A_i U)_{x_i}.
\end{equation}
The entropy $S$ is convex ($S_{UU}=P$) and identical to the energy. A similar identity between the energy and entropy was found in \cite{nordstrom2021linear} for the SWEs. In the following we will use energy to denote both quantities, but sometimes remind the reader by writing out both notations explicitly.
\end{remark}}

\section{The new skew-symmetric form of the compressible Euler equations}\label{sec:skew-eq}
We will go through the list in Remark \ref{list} and start with the choice of norm and new dependent variables.
\subsection{Task 1: Find an appropriate norm and new dependent variables}\label{task_1}
The total energy in a two-dimensional compressible ideal gas is a combination of internal energy $p/(\gamma -1)$, kinetic energy $\rho(u^2+v^2)/2$  and potential energy $\rho g h$. We have $E= p/(\gamma -1) + \rho(u^2+v^2)/2 +  \rho g h$,
where  $p$ is the pressure, $\gamma$ is the ratio of specific heats, $\rho$ is the density, $u,v$ are the velocities in the $(x,y)$ direction respectively and $g h$ is the gravity times the height. Note that $g h$ has the dimension ${velocity}^2$. Based on this observation we choose the dependent variables and the diagonal norm matrix (inspired by E) to be:
\begin{equation}\label{new_var}
\Phi=(\phi_1, \phi_2, \phi_3, \phi_4)^T=(\sqrt{\rho} \textcolor{green}{w}, \sqrt{\rho} u, \sqrt{\rho} v, \sqrt{p})^T, \quad P=diag(\alpha^2, \beta^2, \theta^2, 1), 
\end{equation}
where $\alpha^2, \beta^2, \theta^2$ are positive and non-dimensional but yet unknown. \textcolor{green}{We have also introduced the arbitrary and constant velocity $w$.} The total energy $E$ connects to the quadratic form $\Phi^T P \Phi$ on which we base the norm $\|\Phi\|^2_P$ by observing that all terms involved have dimension  $density \times {velocity}^2$. \textcolor{green}{In the following, we set w=1 without restriction}. This completes the first task in Remark \ref{list}.


\subsection{Task 2: Rewriting the compressible Euler equations in new dependent variables}\label{task_2}

The compressible Euler equations in primitive form using the pressure and the ideal gas law is 
\begin{equation}\label{eq:primitive}
\begin{aligned}
    \rho_t + u \rho_x + v \rho_y +  \rho (u_x+v_y) &=0,\\
    u_t + u u_x + v u_y + p_x/\rho                        &=0,\\
    v_t + u v_x + v v_y + p_y/\rho                         &=0,\\
    p_t + u p_x +v p_y + \gamma p (u_x +v_y)    &=0.\\
\end{aligned}
\end{equation}
Repeated use of the chain rule and the introduction of the new variables $\Phi$ in (\ref{new_var}) leads to the new equations
\begin{equation}\label{eq:skew-symmetric}
\begin{aligned}
    (\phi_1)_t + \frac{u}{2} (\phi_1)_x + \frac{1}{2} (\phi_2)_x + \frac{v}{2} (\phi_1)_y + \frac{1}{2} (\phi_3)_y&=0,\\
    (\phi_2)_t - \frac{u^2}{2} (\phi_1)_x + \frac{3u}{2} (\phi_2)_x + 2 \frac{\phi_4}{\phi_1} (\phi_4)_x - \frac{uv}{2} (\phi_1)_y+ v (\phi_2)_y+ \frac{u}{2} (\phi_3)_y&=0,\\
    (\phi_3)_t - \frac{vu}{2} (\phi_1)_x + \frac{v}{2} (\phi_2)_x + u (\phi_3)_x -\frac{v^2}{2} (\phi_1)_y+ \frac{3v}{2} (\phi_3)_y+2 \frac{\phi_4}{\phi_1} (\phi_4)_y &=0,\\
     (\phi_4)_t - \frac{\gamma}{2} u  \frac{\phi_4}{\phi_1}  (\phi_1)_x + \frac{\gamma}{2} \frac{\phi_4}{\phi_1}  (\phi_2)_x + u (\phi_4)_x  - \frac{\gamma}{2} v  \frac{\phi_4}{\phi_1}  (\phi_1)_y + \frac{\gamma}{2} \frac{\phi_4}{\phi_1}  (\phi_3)_y + v (\phi_4)_y&=0.
\end{aligned}
\end{equation}
For convenience we have used the relations $u=\phi_2/\phi_1, v=\phi_3/\phi_1$. A matrix-vector form of (\ref{eq:skew-symmetric}) is
\begin{equation}\label{eq:skew-symmetric-matrix-vector}
\Phi_t + A \Phi_x + B \Phi_y=0,
\end{equation}
where 
\begin{equation}\label{eq:new_matrix_A,B}
 A =  \frac{1}{2}\begin{bmatrix}
      u                                     & 1     & 0    & 0 \\
     - u^2                                  & 3 u   & 0  &  4 \frac{\phi_4}{\phi_1}\\
     - uv                                   &  v   & 2u & 0                                  \\
     - \gamma u  \frac{\phi_4}{\phi_1}  & \gamma \frac{\phi_4}{\phi_1}     & 0 & 2u
       \end{bmatrix}, \quad 
 B = \frac{1}{2}\begin{bmatrix}
      v                                    & 0      & 1 & 0 \\
       - uv                                   & 2v  & u & 0                                  \\
     - v^2                                  & 3 v   & 0  &  4 \frac{\phi_4}{\phi_1} \\
         - \gamma v  \frac{\phi_4}{\phi_1}  & 0    &  \gamma \frac{\phi_4}{\phi_1}  & 2v
       \end{bmatrix}.
\end{equation}       
This completes the second task in Remark \ref{list}.

\subsection{Task 3: Transforming the new set equations into a skew-symmetric form}\label{task_3}
Proposition \ref{lemma:Matrixrelation} implies that we can proceed each direction separately. By taking into account that we aim for an energy estimate using the (yet unknown) matrix $P$ in (\ref{new_var}), we need to solve
\begin{equation}\label{eq:x-direction_ODE}
(\tilde A  \Phi)_x + \tilde A^T  \Phi_x = 2 P A  \Phi_x,  \quad (\tilde B  \Phi)_y + \tilde B^T  \Phi_y = 2 P B  \Phi_y.
\end{equation}
In (\ref{eq:x-direction_ODE}) we have two systems of ordinary differential equations for the 32 entries $a_{ij}, b_{ij},$ in the matrices $\tilde A, \tilde B$ and the 3 unknowns on the diagonal in $P$.  The number of equations is 8 and we have 4 variables which leads to 32 relations. A simple counting argument implies that the problem is likely solvable. Another important observation is that both the matrices $\tilde A, \tilde B$ and the norm $P$ are coupled and solved for together. To exemplify the procedure we consider the  $x$-direction and later directly provide the results for the $y$-direction. 

We start with row 1 in (\ref{eq:x-direction_ODE}) for the  $x$-direction. The following equation holds
\begin{equation}\label{ex:row_1}
\begin{aligned}
   &(2 a_{11}) (\phi_1)_x + (a_{12}+a_{21})(\phi_2)_x+(a_{13}+a_{31})(\phi_3)_x+(a_{14}+a_{41})(\phi_4)_x\\
+&(a_{11})_x  \phi_1+(a_{12})_x  \phi_2+(a_{13})_x  \phi_3+(a_{14})_x  \phi_4=\alpha^2 u  (\phi_1)_x + \alpha^2 (\phi_2)_x.
\end{aligned}
\end{equation}
There are no entries involving $\phi_3, \phi_4$ on the righthand side (RHS) of the equation. Hence we make the ansatz $a_{13}= \tilde a_{13} \phi_3, a_{31}= \tilde a_{31} \phi_3, a_{14}= \tilde a_{14} \phi_4, a_{41}= \tilde a_{41} \phi_4$. This ansatz  cancels all terms related to $\phi_3, \phi_4$ if $\tilde a_{31}= -2 \tilde a_{13}$ and $\tilde a_{41}= -2 \tilde a_{14}$ with  $\tilde a_{13}, \tilde a_{14}$ as arbitrary constants. The remaining terms on the RHS are
\begin{equation}\label{ex:row_1_RHS}
\alpha^2 u  (\phi_1)_x + \alpha^2 (\phi_2)_x= \alpha^2 ((1-\eta) u (\phi_1)_x + (1+\eta)  (\phi_2)_x - \eta u_x \phi_1),
\end{equation}
with $\eta$ as a free parameter. Equating the terms in (\ref{ex:row_1}) using (\ref{ex:row_1_RHS}) yields $\eta=-1, a_{11}=\alpha^2 u$ \textcolor{blue}{leaving} no terms on the RHS in  (\ref{ex:row_1_RHS}). Hence we get the same type of solution as for $\phi_3, \phi_4$, i.e. $\tilde a_{21}= -2 \tilde a_{12}$. This provide the matrix $\tilde A$ with a determined first row and column as
\begin{equation}\label{eq:new_matrix_tildeA_1}
  \tilde A = \begin{bmatrix}
      \alpha^2 u                                    &  \tilde a_{12} \phi_2 &  \tilde a_{13} \phi_3  & \tilde a_{14} \phi_4 \\
       -2 \tilde a_{12}  \phi_2                          &   \cdot            &   \cdot                       &   \cdot                  \\
       -2 \tilde a_{13}  \phi_3                          &  \cdot             &   \cdot                       &   \cdot                  \\
       -2 \tilde a_{14}  \phi_4                          &  \cdot             &   \cdot                       &   \cdot                  
                   \end{bmatrix}.
 \end{equation}
 
 Next we consider row 2 in (\ref{eq:x-direction_ODE}). The following equation holds
 \begin{equation}\label{ex:row_2}
\begin{aligned}
   &(a_{21}+a_{12})(\phi_1)_x + (2 a_{22})(\phi_2)_x+(a_{23}+a_{32})(\phi_3)_x+(a_{24}+a_{42})(\phi_4)_x\\
+&(a_{21})_x  \phi_1+(a_{22})_x  \phi_2+(a_{23})_x  \phi_3+(a_{24})_x  \phi_4=-\beta^2 u^2  (\phi_1)_x + 3\beta^2 u (\phi_2)_x + 4\beta^2  \frac{\phi_4}{\phi_1} (\phi_4)_x,
\end{aligned}
\end{equation}
where $a_{12}$ and  $a_{21}$ are already determined. Since there are no entries involving $\phi_3$ on the RHS of the equation, we find (as for row 1) that $a_{23}= \tilde a_{23 }\phi_3$ and $a_{32}= -2\tilde a_{32}\phi_3$ is a solution. By inspecting the terms related to $\phi_4$ we see that $a_{24}=0$ and $a_{42}= 4\beta^2  \frac{\phi_4}{\phi_1}$ is the only possible solution. The remaining terms on the RHS of (\ref{ex:row_2}) are rewritten as
\begin{equation}\label{ex:row_2_RHS}
-\beta^2 u^2  (\phi_1)_x + 3\beta^2 u (\phi_2)_x = \beta^2 (3 \eta -1) u^2 (\phi_1)_x + 3 (1- \eta) u (\phi_2)_x + 3 u u_x \phi_1),
\end{equation}
with $\eta$ as a free parameter. By inserting the known values of $a_{12}$ and  $a_{21}$ as well as making the ansatz  $a_{22} = (\tilde a_{22} \phi_1+ \psi_2 u)$ based on the RHS in (\ref{ex:row_2_RHS}) with $\tilde a_{22}, \psi_2$ as free parameters we find the relation
\begin{equation}\label{ex:row_2_RHS_extra}
((\phi_1)_x \phi_2 + 2 (\phi_2)_x \phi_1))  (\tilde a_{22} - \tilde a_{12}) + 2 \psi_2 u (\phi_2)_x + \psi u_x \phi_2  = \beta^2 (3 \eta -1) u^2 (\phi_1)_x + 3 (1- \eta) u (\phi_2)_x + 3 u u_x \phi_1),
\end{equation}
with the solution $\eta = 1/3, \psi_2  = \beta^2, \tilde a_{22} = \tilde a_{12}$. We now have determined also the second row and column
\begin{equation}\label{eq:new_matrix_tildeA_2}
  \tilde A = \begin{bmatrix}
      \alpha^2 u                                    &  \tilde a_{12} \phi_2 &  \tilde a_{13} \phi_3  & \tilde a_{14} \phi_4 \\
       -2 \tilde a_{12}  \phi_2                          &  (\tilde a_{12} \phi_1 +  \beta^2 u)         &   \tilde a_{23} \phi_3                      &   0                \\
       -2 \tilde a_{13}  \phi_3                          &   -2 \tilde a_{23}  \phi_3          &   \cdot                       &   \cdot                  \\
       -2 \tilde a_{14}  \phi_4                          &  4\beta^2  \frac{\phi_4}{\phi_1}            &   \cdot                       &   \cdot                  
                   \end{bmatrix}.
 \end{equation}
The procedure is now clear. One proceeds row by row with a decreasing number of new entries to determine. 

In the third step one makes the ansatz $a_{33} = (\tilde a_{33} \phi_1+ \psi_3 u)$ (for the same reason as in step 2) and finds that $\tilde a_{13} = \tilde a_{23} =\tilde a_{33} =0$ and $\psi_3 = \theta^2$ must hold. The resulting matrix after the third step is
\begin{equation}\label{eq:new_matrix_tildeA_3}
  \tilde A = \begin{bmatrix}
      \alpha^2 u                                    &  \tilde a_{12} \phi_2 &  0  & \tilde a_{14} \phi_4 \\
       -2 \tilde a_{12}  \phi_2                          &  (\tilde a_{12} \phi_1 +  \beta^2 u)         &   0                      &   0                \\
       0                         &   0          &  \theta^2 u                     &   \tilde a_{34} \phi_4              \\
       -2 \tilde a_{14}  \phi_4                          &  4\beta^2  \frac{\phi_4}{\phi_1}            &   -2 \tilde a_{34}  \phi_4                        &   \cdot                
                   \end{bmatrix}.
 \end{equation}
 
 In the fourth step most of the matrix is already determined, and hence we directly state the final equation which after the ansatz $a_{44} = (\tilde a_{44} \phi_1+ \psi_4 u)$ and realising that $\tilde a_{34}=0$ must hold becomes            
\begin{equation}\label{ex:row_4_RHS}
((\phi_1)_x \phi_4 + 2 (\phi_4)_x \phi_1))  (\tilde a_{44} - \tilde a_{14}) + (2 \psi_4 +4 \beta^2) u (\phi_4)_x + (\psi_4 + 4 \beta^2) u_x \phi_4 =  \gamma u_x \phi_4 + 2 u (\phi_4)_x.
\end{equation}
The solution is given by $\beta^2=(\gamma -1)/2, \psi_4= \gamma-2, \tilde a_{44} = \tilde a_{14}$. Note that this gives the first information about the norm matrix $P$ via the requirement for $\beta^2$. The final matrix $\tilde A$  and related norm $P$ becomes
\begin{equation}\label{eq:new_matrix_tildeA_4}
  \tilde A = \begin{bmatrix}
      \alpha^2 u                                    &  \tilde a_{12} \phi_2 &  0 & \tilde a_{14} \phi_4 \\
       -2 \tilde a_{12}  \phi_2                          &  (\tilde a_{12} \phi_1 + \frac{(\gamma -1)}{2}   u)         &  0                   &   0                \\
       0                        &   0        &  \frac{(\gamma -1)}{2}  u                     &  0             \\
       -2 \tilde a_{14}  \phi_4                          &  2 (\gamma -1)  \frac{\phi_4}{\phi_1}            &     0                   & (\tilde a_{14} \phi_1 + (2-\gamma)u)            
                   \end{bmatrix}, \,\
P=\begin{bmatrix}
      \alpha^2    &   0                                  &  0                                 & 0  \\
      0                   &  \frac{(\gamma -1)}{2}   &  0                                 &  0                \\
      0                   &   0                                 &  \frac{(\gamma -1)}{2}  &  0             \\
      0.                  &   0                                 &  0                                 &  1           
                   \end{bmatrix}.
 \end{equation}
To be precise, the third element in $P$ is obtained in the derivation of the matrix $\tilde B$ given below 
\begin{equation}\label{eq:new_matrix_tildeB}
  \tilde B = \begin{bmatrix}
      \alpha^2 v                                    &  0&  \tilde b_{13} \phi_3  & \tilde b_{14} \phi_4 \\
      0                          &  \frac{(\gamma -1)}{2}   v         &   0                     &   0                \\
       -2 \tilde b_{13}  \phi_3                          &  0          &     (\tilde b_{13} \phi_1 + \frac{(\gamma -1)}{2}   v)                &  0             \\
       -2 \tilde b_{14}  \phi_4                          &  0       &    2 (\gamma -1)  \frac{\phi_4}{\phi_1}                    & (\tilde b_{14} \phi_1 + (2-\gamma)v)            
                   \end{bmatrix}.
 \end{equation}
This completes the third task in Remark \ref{list}.

\subsection{Task 4: Applying the energy method such that only boundary terms remain}\label{task_4}

We multiply  (\ref{eq:skew-symmetric-matrix-vector}) with  $2 \Phi^T P$  from the left, use (\ref{eq:x-direction_ODE}) and integrate over the domain $\Omega$. By using Greens formula and Proposition \ref{lemma:Matrixrelation} we find
\begin{equation}\label{eq:boundaryPart_EULER}
\frac{d}{dt}\|\Phi\|^2_P + \oint\limits_{\partial\Omega}^T  \Phi^T (n_x \tilde A + n_y \tilde B) \Phi ds=0,
\end{equation}
where $(n_x,n_y )^T$ is the outward pointing unit normal from the boundary $\partial\Omega$. The relation (\ref{eq:boundaryPart_EULER}) shows that energy  \textcolor{blue}{(and entropy)} is conserved in the sense that it only changes due to boundary effects. 

The matrices $\tilde A, \tilde B$ and the norm matrix $P$ contain 5 undetermined parameters $\tilde a_{12}, \tilde a_{14}, \tilde b_{13},\tilde b_{14}$ and  $\alpha^2$.
Except for the parameter $\alpha^2$, which is part of the norm, the energy rate cannot depend on these parameters since they are not present in (\ref{eq:skew-symmetric-matrix-vector}), (\ref{eq:new_matrix_A,B}) and (\ref{eq:x-direction_ODE}).  Hence as a sanity check we compute the boundary contraction involving only the terms multiplied by $\tilde a_{12}, \tilde a_{14}, \tilde b_{13},\tilde b_{14}$ and find
\begin{equation}\label{boundarmatrix_free_contraction}
\Phi^T \begin{bmatrix}
       0                    &  n_x\tilde a_{12} \phi_2   &  n_y \tilde b_{13} \phi_3     & (n_x \tilde a_{14} + n_y \tilde b_{14}) \phi_4                           \\
       -2 n_x \tilde a_{12}  \phi_2     &  n_x \tilde a_{12} \phi_1         &  0                                          &   0                                                     \\
       -2 n_y \tilde b_{13}  \phi_3    &   0                                                                                     &  n_y \tilde b_{13} \phi_1 &  0                                                      \\
       -2 (n_x \tilde a_{14} + n_y \tilde b_{14})  \phi_4     &  0  & 0 & (n_x \tilde a_{14} + n_y \tilde b_{14})\phi_1        
                   \end{bmatrix} \Phi=0,
\end{equation}
showing that the free parameters do not influence the energy rate.
The remaining boundary contraction is
\begin{equation}\label{boundarmatrix_contraction}
\Phi^T (n_x \tilde A + n_y \tilde B) \Phi = 
\Phi^T  \begin{bmatrix}
      \alpha^2 u_n  &  0 &  0 & 0                       \\
       0                    &  \frac{(\gamma -1)}{2}   u_n         &  0                                          &      n_x (\gamma -1)  \frac{\phi_4}{\phi_1}                                                 \\
       0                    &   0                                                                                     &  \frac{(\gamma -1)}{2}   u_n   &   n_y (\gamma -1)  \frac{\phi_4}{\phi_1}                                                      \\
       0                    &  n_x (\gamma -1)  \frac{\phi_4}{\phi_1}& n_y (\gamma -1)  \frac{\phi_4}{\phi_1} & (2-\gamma)u_n            
                   \end{bmatrix} \Phi,
\end{equation}
where we  introduced the normal velocity $u_n=n_xu+n_yv$. This completes the fourth task in Remark \ref{list}.

\subsection{Task 5: The choice of nonlinear boundary conditions }\label{task_5}
A nonlinear and linear analysis may lead to a different number and type of boundary conditions required for an energy bound. This was discussed extensively in \cite{nordstrom2021linear}, \cite{nordstrom2022linear-nonlinear} where the boundary matrix was found to be different in the linear and nonlinear case, and also to have a different meaning. For completeness we will repeat part of that discussion here. For more details we refer the reader to \cite{nordstrom2021linear}, \textcolor{blue}{and} \cite{nordstrom2022linear-nonlinear}.

We start by rotating the velocities to be normal ($u_n=n_xu+n_yv$) or aligned ($u_{\tau} =-n_yu+n_xv$) with the boundary. By inserting these transformation into (\ref{boundarmatrix_contraction}) we obtain the rotated boundary contraction 
\begin{equation}\label{boundarmatrix_contraction_rotated}
\Phi^T (n_x \tilde A + n_y \tilde B) \Phi = 
\Phi^T_r  \begin{bmatrix}
      \alpha^2 u_n  &  0 &  0 & 0                       \\
       0                    &  \frac{(\gamma -1)}{2}   u_n         &  0                                          &      (\gamma -1)  \frac{\phi_4}{\phi_1}                                                 \\
       0                    &   0                       &  \frac{(\gamma -1)}{2}   u_n                                                                 &  0                                                      \\
       0                    &  (\gamma -1)  \frac{\phi_4}{\phi_1}& 0 & (2-\gamma)u_n            
                   \end{bmatrix} \Phi_r,
\end{equation}
where the rotated solution is $\Phi_r=(\phi_1, \phi_1 u_n, \phi_1 u_{\tau},  \phi_4)^T$. By considering the eigenvalues of the matrix we see that they indicate a similar but not identical sign pattern as in the linear case. We find the eigenvalues
\begin{equation}\label{eigenvalues_rotated}
\lambda_1= \alpha^2 u_n,  \quad \lambda_2=  \frac{(\gamma -1)}{2} u_n, \quad  \lambda_{3,4}=\frac{(3-\gamma)}{4} u_n   \pm  \sqrt{ \left(\frac{(3-\gamma)}{4} u_n\right)^2-a^2 c^2((M_n^2-b^2)},
\end{equation}
where $a^2=\frac{(\gamma-1)(2-\gamma)}{2}, b^2 = \frac{2(\gamma-1)}{\gamma(2-\gamma)}$, $c$ is the speed of sound and $M_n = u_n/c$, the normal Mach number.

The relations (\ref{eigenvalues_rotated}) indicate that for outflow ($u_n > 0$) we have two situations. When $M_n > b$ there are 4 positive eigenvalues and no boundary condition is required. For $M_n < b$, 3 eigenvalues are positive, 1 is negative and 1 boundary condition seem to be required. For inflow ($u_n < 0$) we have the reversed situation with 4 negative eigenvalues and four required boundary conditions for $M_n > b$, which goes down to 3 negative eigenvalues and 3 required boundary conditions for $M_n < b$. 
\begin{remark}
The shift from subsonic to supersonic flow at $M_n =1$ is generally assumed to be crucial and to modify the number of boundary conditions in a linear analysis. Interestingly, here in the nonlinear analysis the shift occur when $b \approx 0.976$ if $\gamma=1.4$, Maybe even more interesting is that $b \equiv 1$ for $\gamma = \sqrt{2}  \approx 1.414$.
\end{remark}

However, considering eigenvalues is not sufficient for nonlinear problems  \cite{nordstrom2021linear},\cite{nordstrom2022linear-nonlinear}. Expanding (\ref{boundarmatrix_contraction_rotated}) give
\begin{equation}\label{boundarmatrix_contraction_rotated_expanded}
\Phi^T_r  \begin{bmatrix}
      \alpha^2 u_n  &  0 &  0 & 0                       \\
       0                    &  \frac{(\gamma -1)}{2}   u_n         &  0                                          &      (\gamma -1)  \frac{\phi_4}{\phi_1}                                                 \\
       0                    &   0                       &  \frac{(\gamma -1)}{2}   u_n                                                                 &  0                                                      \\
       0                    &  (\gamma -1)  \frac{\phi_4}{\phi_1}& 0 & (2-\gamma)u_n            
                   \end{bmatrix} \Phi_r=u_n \left( \alpha^2 \phi_1^2+\frac{(\gamma -1)}{2}(\phi_2^2+\phi_3^2) + \gamma \phi_4^2\right),
\end{equation}
which proves that no boundary conditions are necessary in the outflow case. It also proves that specifying the normal velocity to zero is correct at a solid wall,  \textcolor{blue}{see \cite{svard2014,parsani2015entropy,svar2018,svard2021,chan2022} for some previous results on this matter.}
This completes the fifth and final task in Remark \ref{list}.


\subsection{The final form of the governing equations}\label{final_eq}
The derivations above focused on stability and resulted in matrices $\tilde A, \tilde B$ that were functions of the constant norm matrix $P$ as seen in (\ref{eq:x-direction_ODE}). 
The final governing equations can be transformed to
\begin{equation}\label{eq:stab_eq}
\Phi_t + (A_1  \Phi)_x + A_2  \Phi_x + (B_1  \Phi)_y + B_2  \Phi_y=0,
\end{equation}
where 
\begin{equation}\label{eq:stab_eq_matrixdef}
A_1=P^{-1} \tilde A/2, \quad A_2= P^{-1} \tilde A^T/2,  \quad B_1= P^{-1} \tilde B/2, \quad  B_2 = P^{-1} \tilde B^T/2. 
\end{equation}
This removes the dependence of the norm $P$ in the matrices which take the form (without  free parameters)
\begin{equation}\label{eq:new_matrix_finalAs}
A_1 = \frac{1}{2}\begin{bmatrix}
      u                                    &  0  &  0 & 0 \\
      0                                    &  u  &  0 & 0                \\
      0  &   0                                                      &  u &  0             \\
      0  &  2 (\gamma -1)  \frac{\phi_4}{\phi_1} & 0  &  (2-\gamma)u           
                   \end{bmatrix}, \quad
A_2 = \frac{1}{2}\begin{bmatrix}
      u                                    &  0  &  0 & 0 \\
      0                                    &  u  &  0 & 4  \frac{\phi_4}{\phi_1}                \\
      0  &   0                                                      &  u &  0             \\
      0  &   0 & 0  &  (2-\gamma)u           
                   \end{bmatrix},                    
 \end{equation}
 \begin{equation}\label{eq:new_matrix_finalBs}
B_1 = \frac{1}{2}\begin{bmatrix}
      v                                    &  0  &  0 & 0 \\
      0                                    &  v  &  0 & 0                \\
      0  &   0                                                      &  v &  0             \\
      0  &  0 & 2 (\gamma -1)  \frac{\phi_4}{\phi_1}   &  (2-\gamma)v          
                   \end{bmatrix}, \quad
B_2 = \frac{1}{2}\begin{bmatrix}
      v                                   &  0  &  0 & 0 \\
      0                                    &  v  &  0 & 0               \\
      0  &   0                                                      &  v &  4  \frac{\phi_4}{\phi_1}           \\
      0  &   0 & 0  &  (2-\gamma)v        
                   \end{bmatrix}.                   
 \end{equation}
 
\subsection{Some open questions}\label{open_questions}
It is interesting to consider the energy \textcolor{blue}{(and entropy)} rate. By combining  (\ref{eq:boundaryPart_EULER}) and (\ref{boundarmatrix_contraction_rotated_expanded})
we find
\begin{equation}\label{eq:boundaryPart_EULER_rate}
\frac{d}{dt} \int\limits_{\Omega}    \Phi^T P \Phi  dxdy + \oint\limits_{\partial\Omega}  u_n( \Phi^T P \Phi  + (\gamma -1) p) ds=0.
\end{equation}
This means the rate of change in the domain $\Omega$ of the energy ($ \Phi^T P \Phi $) is increasing or decreasing due to the transport of energy in or out of the domain plus an additional amount due to pressure work 
$(\gamma -1) u_n p$.

As we have seen the $\tilde a_{12}, \tilde a_{14}, \tilde b_{13},\tilde b_{14}$ does not contribute to the rate of change in the energy but could be part of the scheme by populating the matrices in (\ref{eq:new_matrix_finalAs}) and (\ref{eq:new_matrix_finalBs}). It is an open question whether they will modify the spectrum and hence time-integration procedure. The parameter $\alpha^2$ in the norm could be any positive number that defines a reasonable norm. It has no influence on the scheme.

There are two possible interpretations of the sign requirements for the density $\rho$ and the pressure $p$ (and hence temperature). The first interpretation considers the problem in a physical way which means that $\rho$ and $p$ must be positive, otherwise the new variables involving square roots do not exist. In the second opposite interpretation, the new variables are considered as the ones defining the original variables. With this point of view, the sign problem vanishes since by squaring $\phi_1$ and $\phi_4$ both  $\rho$ and $p$ will always be positive.  \textcolor{blue}{The bounds on the new variables directly lead to bounds on the original variables, by squaring them}.

\section{A stable energy \textcolor{blue}{and entropy} conserving numerical approximation}\label{numerics}
To exemplify the straightforward construction of stable schemes based on the new formulation, we consider a summation-by-parts (SBP) approximation of (\ref{eq:stab_eq}),(\ref{eq:stab_eq_matrixdef}) as given in
\begin{equation}\label{EUL_Disc}
\vec \Phi_t+{\bf D_x} ({\bf A_1}  \vec \Phi)+{\bf A_2} {\bf D_x}  (\vec \Phi) +{\bf D_y} ({\bf B_1}  \vec \Phi)+{\bf B_2} {\bf D_y}  (\vec \Phi)=0,
\end{equation}
where $\vec \Phi=(\vec \Phi_1^T, \vec \Phi_2^T,...,\vec \Phi_n^T)^T$ include approximations of  $\Phi=(\phi_1,\phi_2,...,\phi_n)^T$ in each node.  The matrix elements of ${\bf A_1},{\bf A_2},{\bf B_1},{\bf B_2}$ are matrices with node values of the matrix elements in $A_1,A_2,B_1,B_2$ injected on the diagonal as exemplified below
\begin{equation}
\label{illustration}
A_1 =
\begin{pmatrix}
      a_{11}   &  \ldots  & a_{1n} \\
       \vdots   & \ddots & \vdots \\
       a_{n1} &  \ldots  & a_{nn}
\end{pmatrix}, \quad
{\bf A_1} =
\begin{pmatrix}
      {\bf a_{11}}   &  \ldots  &  {\bf a_{1n}}  \\
       \vdots          & \ddots &  \vdots           \\
        {\bf a_{n1}} &  \ldots &  {\bf a_{nn}} 
\end{pmatrix}, \quad
{\bf a_{ij}} =diag(a_{ij}(x_1,y_1), \ldots, a_{ij}(x_N,y_M)).
\end{equation}

Moreover ${\bf D_x}=I_n \otimes D_{x} \otimes I_y$ and ${\bf D_y}=I_n \otimes I_x \otimes D_y$ where $D_{x,y}=P^{-1}_{x,y}Q_{x,y}$ are 1D SBP difference operators, $P_{x,y}$ are positive definite diagonal quadrature matrices, $Q_{x,y}$ satisfies the SBP constraint $Q_{x,y}+Q_{x,y}^T=B_{x,y}=diag[-1,0,...,0,1]$, $\otimes$ denotes the Kronecker product  and $I$ with subscripts denote identity matrices. All matrices have appropriate sizes such that the matrix-matrix and matrix-vector operations are defined. 
Based on the 1D SBP operators, the 2D SBP relations mimicking integration by parts are given by 
\begin{equation}\label{Multi-SBP}
\vec U^T  \tilde {\bf P} {\bf D_x} \vec V= -({\bf D_x} \vec U)^T \tilde {\bf P} \vec V + \vec U^T {\bf B_x} \vec V, \quad 
\vec U^T \tilde {\bf P} {\bf D_y} \vec V= -({\bf D_y} \vec U)^T  \tilde {\bf P} \vec V + \vec U^T {\bf B_y} \vec V,
\end{equation}
where  $\vec U^T {\bf B_x} \vec V$ and $\vec U^T {\bf B_y} \vec V$ contain numerical integration along rectangular domain boundaries. In (\ref{Multi-SBP}) we have used
$ \tilde {\bf P}=I_n  \otimes P_x   \otimes P_y$, ${\bf B_x}=(I_n \otimes B_x \otimes P_y)$ and ${\bf B_y}=(I_n \otimes P_x \otimes B_y)$. 

The discrete energy method (multiply (\ref{EUL_Disc}) from the left with  $2 \vec \Phi^T {\bf P} \tilde {\bf P}$) where ${\bf P} = P \otimes I_x   \otimes I_y$ yields 
\begin{equation}\label{SWE_Disc_energy}
2 \vec \vec \Phi^T ({\bf P} \tilde {\bf P}) \vec \Phi_t+ \vec \Phi^T  \tilde {\bf P} {\bf D_x} (2{\bf P}{\bf A_1})  \vec \Phi+ \vec \Phi^T \tilde {\bf P} (2  {\bf P}{\bf A_2}){\bf D_x}  (\vec \Phi) + 
                                                              \vec \Phi^T  \tilde {\bf P} {\bf D_y} (2 {\bf P} {\bf B_1})  \vec \Phi+ \vec \Phi^T \tilde {\bf P} (2 {\bf P} {\bf B_2}) {\bf D_Y}  (\vec \Phi) =0,
\end{equation}
where we have used that ${\bf P}$ commutes with ${\bf D_x},{\bf D_y}$. Next, the discrete relations corresponding to (\ref{eq:stab_eq_matrixdef}),
\begin{equation}\label{eq:stab_eq_matrixdef_disc}
\bf A_1={\bf P}^{-1} \tilde {\bf A}/2, \quad A_2= P^{-1} \tilde A^T/2,  \quad B_1= P^{-1} \tilde B/2, \quad  B_2 = P^{-1} \tilde B/2,
\end{equation}
and the notation  $\| \vec U\|_{{\bf P} \tilde {\bf P}}^2=\vec U^T ({\bf P} \tilde {\bf P}) \vec U$ transforms (\ref{SWE_Disc_energy}) to
\begin{equation}\label{SWE_Disc_energy_final}
\dfrac{d}{dt} \|\Phi \vec \|_{{\bf P}  \tilde {\bf P}}^2+ \vec \Phi^T  \tilde {\bf P} {\bf D_x} ( \tilde {\bf A})  \vec \Phi+ \vec \Phi^T \tilde {\bf P} ( \tilde {\bf A}^T){\bf D_x}  (\vec \Phi) + 
                                                              \vec \Phi^T  \tilde {\bf P} {\bf D_y} ( \tilde {\bf B})  \vec \Phi+ \vec \Phi^T \tilde {\bf P} ( \tilde {\bf B}^T) {\bf D_y}  (\vec \Phi) =0.
\end{equation}
Finally, the SBP relations (\ref{Multi-SBP}), $\tilde {\bf P} \tilde {\bf A}=\tilde {\bf A} \tilde {\bf P}$ and  $\tilde {\bf P} \tilde {\bf B} =\tilde {\bf B}\tilde {\bf P}$ (the matrices consist of diagonal blocks) yield
 \begin{equation}\label{Disc_energy_final}
\dfrac{d}{dt} \| \Phi \vec \|_{{\bf P}  \tilde {\bf P}}^2 + \vec \Phi^T( {\bf B_x} (\tilde {\bf A}) + {\bf B_y} (\tilde {\bf B})  \vec \Phi=0.
\end{equation}
The semi-discrete energy rate in (\ref{Disc_energy_final}) mimics the continuous result in (\ref{eq:boundaryPart_EULER}) and hence the scheme is energy  \textcolor{blue}{and entropy} conserving.  Stability can be obtained by adding a proper dissipative boundary treatment. 
\begin{remark}
It is irrelevant whether the problem is linear or nonlinear. The skew-symmetric formulation, an SBP discretisation and a proper boundary treatment \textcolor{green}{are all that is needed for stability}.
\end{remark}
\section{Summary, \textcolor{green}{conclusions and outlook}}\label{sec:conclusion}
We have shown that a specific skew-symmetric form of linear and nonlinear problem leads to \textcolor{green}{energy and entropy bounds} for the compressible Euler equations. The skew-symmetric formulation automatically produced \textcolor{green}{energy and entropy} stable numerical schemes for the compressible Euler equations if these are formulated on summation-by-parts form. 

The derivation shoved that the skew-symmetric formulation required a coupled derivation of the matrices and the norm matrix. The derivations also indicated that the shift in number of boundary conditions might not occur precisely at Mach number = 1, but at a slightly lower number given by $2(\gamma-1)/(\gamma(2-\gamma))$ for $\gamma=1.4$. It also shoved that this ratio is identically one for  $\gamma= \sqrt{2} \approx 1.414$.

This paper together with  \cite{nordstrom2022linear-nonlinear} have shown that the incompressible Euler equations, the shallow water equations and the compressible Euler equations can all be transformed to skew-symmetric form. Once in that form stable, easy to apply nonlinear schemes follows if summation--by parts operators are used for the discretisation. No additional requirements, such as chain rules or sign requirements are needed.

\textcolor{green}{In future work we will continue the study of nonlinear boundary conditions, and especially it's relation to linear boundary procedures. We will also investigate presently unknown numerical advantages and disadvantages, such as stiffness, robustness, coarse mesh effects etc. In addition we will include dissipative effects, stemming from viscous terms in the Euler case, and bottom effects for the SWEs.}

\section*{Acknowledgments}

Many thanks to my colleagues Fredrik Laur{\'e}n and Andrew R. Winters for helpful comments on the manuscript. Jan Nordstr\"{o}m was supported by Vetenskapsr{\aa}det, Sweden [award no.~2018-05084 VR and 2021-05484 VR] and the Swedish e-Science Research Center (SeRC).

\bibliographystyle{elsarticle-num}
\bibliography{References_andrew,References_Fredrik}

\end{document}